\documentclass[11pt,a4paper]{article}

\usepackage[T1]{fontenc}
\usepackage[utf8]{inputenc}
\usepackage[english]{babel}
\usepackage{amsmath,amssymb,amsthm,mathtools,comment}
\usepackage{booktabs,tabularx,array}
\usepackage{enumitem}
\usepackage{microtype}
\usepackage{xcolor}
\usepackage{graphicx}
\usepackage{adjustbox}
\usepackage{float}
\usepackage{pdflscape}
\usepackage[a4paper,margin=28mm,headheight=14pt]{geometry}
\usepackage[numbers,sort&compress]{natbib}
\usepackage{hyperref}
\usepackage[capitalise,noabbrev]{cleveref}
\usepackage{tikz}
\usepackage{tikz-cd}
\usetikzlibrary{arrows.meta,positioning,fit,backgrounds,calc}

\hypersetup{
  colorlinks=true,
  linkcolor=blue!45!black,
  citecolor=green!35!black,
  urlcolor=blue!55!black,
  pdftitle={On Weak Set Theories Interpreted in PA},
  pdfauthor={Junhong Chen}
}

\setlist{itemsep=2pt,topsep=5pt}
\allowdisplaybreaks

\newcommand{\N}{\mathbb N}

\newcommand{\Seq}{\operatorname{Seq}}

\newcommand{\Q}{\mathsf Q}
\newcommand{\PA}{\mathsf{PA}}
\newcommand{\PAminus}{\mathsf{PA}^{-}}
\newcommand{\IDelta}{I\Delta_{0}}
\newcommand{\IS}[1]{I\Sigma_{#1}}
\newcommand{\BS}[1]{B\Sigma_{#1}}
\newcommand{\expax}{\mathsf{exp}}
\newcommand{\OmegaOne}{\Omega_{1}}

\newcommand{\Szero}{\mathsf S_{0}}
\newcommand{\AD}{\mathsf{AD}}
\newcommand{\ReSzero}{\mathsf{ReS}_{0}}
\newcommand{\DB}{\mathsf{DB}_{0}}
\newcommand{\GJ}{\mathsf{GJ}_{0}}
\newcommand{\Tzero}{\mathsf T_{0}}

\newcommand{\Coll}[1]{\mathsf{Coll}(#1)}
\newcommand{\SColl}[1]{\mathsf{Coll}_{s}(#1)}
\newcommand{\Sep}[1]{\mathsf{Sep}(#1)}
\newcommand{\Repl}[1]{\mathsf{Repl}(#1)}
\newcommand{\Fnd}[1]{#1\text{-}\mathsf{Foundation}}
\newcommand{\SetFnd}{\mathsf{SetFoundation}}
\newcommand{\Pow}{\mathsf{PowerSet}}
\newcommand{\Inf}{\mathsf I}
\newcommand{\TCl}{\mathsf{TCl}}

\newcommand{\Int}{\mathrel{\triangleright}}
\newcommand{\Mint}{\mathrel{\equiv_{\mathrm{int}}}}
\newcommand{\Dedeq}{\mathrel{\equiv_{\mathrm{ded}}}}

\newcolumntype{Y}{>{\raggedright\arraybackslash}X}

\theoremstyle{plain}
\newtheorem{theorem}{Theorem}[section]
\newtheorem{lemma}[theorem]{Lemma}

\theoremstyle{definition}
\newtheorem{definition}[theorem]{Definition}

\theoremstyle{remark}

\tikzset{
  theory/.style={
    draw=black!65,
    fill=white,
    rounded corners=1pt,
    align=center,
    inner xsep=4pt,
    inner ysep=3pt,
    font=\scriptsize
  },
  ded/.style={-{Latex[length=1.8mm]},semithick,black!80},
  same/.style={densely dashed,black!55,semithick},
  degreebox/.style={
    draw=black!45,
    fill=black!2,
    rounded corners=2pt,
    inner sep=6pt
  }
}

\title{On Weak Set Theories Interpreted in \(\mathsf{PA}\)}
\author{Junhong Chen}
\date{August 2026}

\begin{document}

\maketitle

\begin{abstract}
We classify a broad family of weak first-order set theories, under ordinary parameter-free interpretability, by the first-order arithmetical theories with which they are mutually interpretable. This also determines their consistency strength. Set theories that correspond to the same arithmetical theory are often related by deductive extension. It is therefore enough to interpret a stronger set theory in the arithmetical theory and to recover the arithmetical theory in a weaker set theory; all intermediate cases then follow. The set theories under consideration fall roughly into three classes: theories with neither Power Set nor Infinity, theories with Power Set but without Infinity, and theories with Infinity but without Power Set. We conclude with a brief account of the higher levels that remain to be investigated.
\end{abstract}


\section{Introduction}

G\"odel's discussion of the continuum problem placed the study of set-theoretic axioms at the centre of foundational research \citep{Godel1947}. Weak set theories and their basic properties have since been studied systematically; see, for example, \citep{Mathias2001,Mathias2007}. This paper makes an exact comparison between set theory and the weak end of arithmetic. Given a weak set theory \(T\), how much arithmetic is needed to define a model of \(T\)? Conversely, how much arithmetic can be recovered inside \(T\)?

In keeping with this aim, we use ordinary first-order interpretability. The interpretations of finite set theory developed by \citet{KayeWong2007,Pettigrew2009} are important precedents for this form of comparison. Ordinary interpretability is strictly stronger than equiconsistency. An external model construction, a forcing extension, or a model obtained from a generic parameter does not by itself provide a fixed family of parameter-free interpreting formulas. At the same time, interpretability is of course coarser than deductive extension. This observation determines the organisation of our proofs: once interpretability has been established at the strongest and weakest endpoints, all theories between them can be placed in the same mutual-interpretability class at once.

Proof-theoretic ordinals provide another useful measure, but they are not automatically preserved by ordinary interpretations. An interpretation may replace the standard natural numbers by a definable initial segment or a quotient, and it may change formula complexity. Transferring an ordinal analysis across an interpretation therefore requires an additional preservation theorem, followed by a separate verification that each of our interpretations meets its hypotheses. Moreover, the systems considered here are far weaker than the theories for which ordinal analysis is usually informative, so merely identifying their proof-theoretic ordinals would add little to the present classification. We do not pursue that question here.

\section{Preliminaries}

\begin{definition}
Let \(S\) and \(T\) be first-order theories. We write \(T\Int S\) if a fixed finite family of parameter-free formulas in the language of \(T\) defines, in every model of \(T\), a nonempty quotient structure satisfying \(S\). These formulas specify the domain, an equivalence relation on that domain, and the interpretations of the relations and functions of \(S\). If both \(T\Int S\) and \(S\Int T\), we write \(T\Mint S\).
\end{definition}

Thus an interpretation arrow points from the interpreting theory to the interpreted theory. In the deductive diagrams below, however, a solid arrow \(U\longrightarrow V\) means \(U\vdash V\), and hence points from the stronger theory to the weaker theory in the same language.

For set-theoretic formulas we use the L\'evy hierarchy. A \(\Delta_0\) formula contains only quantifiers of the forms \(\forall x\in a\) and \(\exists x\in a\), while \(\Sigma_n\) and \(\Pi_n\) are obtained by strict alternation of unbounded quantifier blocks in prenex form. This is the convention used for the relevant comparisons of Collection and Replacement in \citep{Mathias2007}. We use the standard arithmetical hierarchy in the language of arithmetic.

Fix a formula \(\varphi(x,y,\bar p)\), where \(x\) is the input, \(y\) is a witness, and \(\bar p\) is a tuple of parameters; let \(a\) be the set over which witnesses are to be collected. The corresponding instance \(\Coll{\varphi}\) of Collection is
\begin{equation}
  \forall x\in a\,\exists y\,\varphi(x,y,\bar p)\longrightarrow
  \exists b\,\forall x\in a\,\exists y\in b\,\varphi(x,y,\bar p),
\label{eq:ordinary-collection}
\end{equation}
whereas the corresponding instance \(\SColl{\varphi}\) of Strong Collection is
\begin{equation}
  \exists b\,\forall x\in a\,
  \bigl(\exists y\,\varphi(x,y,\bar p)\longleftrightarrow
  \exists y\in b\,\varphi(x,y,\bar p)\bigr).
\label{eq:strong-collection}
\end{equation}
The corresponding instance \(\Repl{\varphi}\) of Replacement is
\begin{equation}
  \forall x\in a\,\exists! y\,\varphi(x,y,\bar p)\longrightarrow
  \exists b\,\forall y\,
  \bigl(y\in b\longleftrightarrow\exists x\in a\,\varphi(x,y,\bar p)\bigr).
\end{equation}
Here \(b\) is the exact range of the function on \(a\). For a formula \(A(x,\bar p)\), where \(x\) is a candidate minimal element and \(\bar p\) is a tuple of parameters, the corresponding instance \(\Fnd{A}\) of Foundation is
\begin{equation}
  \exists x\,A(x,\bar p)\longrightarrow
  \exists x\bigl(A(x,\bar p)\wedge\forall y\in x\,\neg A(y,\bar p)\bigr).
\label{eq:foundation}
\end{equation}
Writing a formula class \(\Gamma\) in parentheses denotes the scheme containing the corresponding instance for every formula in \(\Gamma\).

The following facts are standard.
\begin{lemma}[Implications between schemes]
\label{lem:scheme-implications}
Over \(\ReSzero\):
\begin{enumerate}[label=\textup{(\roman*)}]
\item \(\SColl{\Delta_0}\Dedeq\Coll{\Delta_0}+\Sep{\Sigma_1};\)
\item for every \(n\geq1\), \(\SColl{\Sigma_n}\Dedeq\Coll{\Sigma_n}+\Sep{\Sigma_n};\)
\item for \(n\geq1\), \(\Coll{\Sigma_{n+1}}\) implies \(\SColl{\Sigma_n}\);
\item for \(n\geq1\), \(\SColl{\Sigma_n}\) implies \(\Repl{\Sigma_n}\);
\item for \(n\geq1\), \(\Repl{\Sigma_{n+1}}\) implies \(\SColl{\Sigma_n}\).
\end{enumerate}
\end{lemma}

Proofs of these implications may be found in \citet{Mathias2007} and \citet{McKenzie2019}. We shall use only the directions displayed above.

All our set theories are formulated in the first-order language \(\{\in\}\).
\begin{definition}
The theory \(\Szero\) consists of Extensionality, Empty Set, Pairing, Set Difference, and Union. Let
\[
 \ReSzero=\Szero+\Sep{\Delta_0},
\]
which is well known to be finitely axiomatizable. Define \(\DB\) by
\[
 \DB\Dedeq\ReSzero+\mathsf{CartesianProduct}.
\]
The theory \(\GJ\) further includes the elementary image-of-sections operation
\[
 R_8(x,y)=\{x^{\prime\prime}\{w\}:w\in y\},
 \qquad x^{\prime\prime}\{w\}=\{u:\langle u,w\rangle\in x\}.
\]
Finally, let
\[
 \Tzero=\GJ+\SetFnd+\TCl.
\]
Here \(\TCl\) asserts that every set is contained in a least transitive set. Clearly \(\ReSzero\) proves \(\SetFnd\leftrightarrow\Fnd{\Delta_0}\), and we shall use the two notations interchangeably. We write \(\Inf\) for Strong Infinity, which asserts that the class of finite von Neumann ordinals is a set \(\omega\); it is well known that both \(x=\omega\) and \(x\in\omega\) are \(\Delta_0\). We write \(\Pow\) for Power Set.
\end{definition}

The operation \(R_8\) collects the sections of one fixed elementary relation. It is not equivalent to an arbitrary instance of Replacement or Collection; the deductive distinctions between these schemes are analysed in \citep{Mathias2007}.

We also use standard notation for arithmetical theories.

\begin{theorem}[The arithmetic spine]
With respect to ordinary interpretability,
\begin{align}
\Q&\Mint\PAminus
\Mint\IDelta+\OmegaOne
\Mint\BS1+\OmegaOne,                                      \label{eq:q-spine}\\
\IDelta+\expax&\Mint\BS1+\expax,                         \label{eq:exp-spine}\\
\BS{m}&\Mint I\Sigma_{m-1}\qquad(m\geq1),                \label{eq:ind-spine}
\end{align}
where \(I\Sigma_0\) in \eqref{eq:ind-spine} means \(I\Delta_0\). The degrees represented by
\[
 \Q,\quad\IDelta+\expax,\quad
 I\Sigma_1,\quad I\Sigma_2,\ldots,\quad\PA
\]
are strictly increasing.
\end{theorem}

The initial-segment interpretations used in \eqref{eq:q-spine} are given in \citet{Pudlak1985,Visser2017}. For \eqref{eq:exp-spine}, \eqref{eq:ind-spine}, and the strictness of the hierarchy, see \citet{Hajek1993Interpretability,Cheng2021}.

We have not found the following natural arithmetical interpretation in the literature, so we give the proof. It is well known that \(\PAminus+\mathsf{Coll}\) is \(\Pi_1\)-conservative over \(\PAminus\), and hence has the same consistency strength. In fact, the stronger theory is still interpretable in \(\Q\).
\begin{theorem}
  \(\Q\) interprets \(\PAminus+\mathsf{Coll}\).
\end{theorem}
\begin{proof}
  Work first in a convenient bounded arithmetical theory interpretable in \(\Q\), such as \(\mathsf{B}\Sigma_1+\Omega_1\). Represent a rational number by three natural numbers: a numerator, a nonzero denominator, and a sign. Standard \(\Delta_0\)-definable pairing and projection functions represent these three entries by one number. Equality and order on the resulting rationals are likewise \(\Delta_0\)-definable.

  We represent the formal polynomial ring \(\mathbb Z[X_q]_{q\in\mathbb Q}\). A term \(X_{q_1}\cdots X_{q_n}\) is represented by a finite nonincreasing sequence \(q_1\geq\cdots\geq q_n\). A monomial is such a term together with an integer coefficient, and a polynomial is a finite sequence of monomials in decreasing order. All the required representations and projections are \(\Delta_0\)-definable, and equality and the ring operations on formal polynomials are computable in polynomial time. Order the terms lexicographically: set
  \[
  \langle q_1,\ldots,q_n\rangle>
  \langle q'_1,\ldots,q'_m\rangle
  \]
  if, at their first point of difference, \(q_i>q'_i\), or if the second sequence is a proper initial segment of the first. Declare a nonzero polynomial positive precisely when the coefficient of its largest term is positive. The nonnegative polynomials then form a model of \(\PAminus\).

  It remains to verify \(\mathsf{Coll}\). Suppose
  \(\forall X<A\,\exists Y\,\varphi(X,Y,\bar P)\). Only finitely many variables \(X_q\) occur in \(A\) and \(\bar P\), so there is a rational \(q_0\) such that both \(A\) and every member of \(\bar P\) are smaller than \(X_{q_0}\). Choose \(q_1>q_0\). The ordered field of rationals has an order isomorphism from \((q_0,+\infty)\) onto \((q_0,q_1)\). It induces an automorphism of the formal polynomial ring and fixes every polynomial whose variables have indices at most \(q_0\). Applying this automorphism to witnesses shows
  \[
  \forall X<A\,\exists Y<X_{q_1}\,\varphi(X,Y,\bar P),
  \]
  which is the required instance of Collection.
\end{proof}

\section{Interpreting Set Theory in Arithmetic}
\label{sec:upper}

We begin with a representation that will be used throughout this section. A rooted finite directed acyclic graph is written as \(G=(N,E,r)\), where \(N>0\) is the number of vertices, \(r<N\) is the root, \(E\) is a relation on \(N\), and
\[
 E(i,j)\longrightarrow j<i.
\]
The free variables \(i\) and \(j\) here range over vertices below \(N\); the formula says that every edge points to a vertex with a smaller number. In arithmetic, the relation \(E\) is represented by one number, accompanied by the finite sequence
\(1,2,4,\ldots,2^K\), where \(K\geq N^2\). Whether a given bit is one can then be decided by a bounded formula, and the existence of the finite sequence requires only \(\Omega_1\).

Suppose \(G=(N_G,E_G,r_G)\) and \(H=(N_H,E_H,r_H)\) are two such graphs, and let \(B\) be a relation on \(N_G\times N_H\). For vertices \(i<N_G\) and \(j<N_H\), require
\begin{align*}
B(i,j)&\longrightarrow
 \forall i'<N_G\bigl(E_G(i,i')\to
   \exists j'<N_H(E_H(j,j')\wedge B(i',j'))\bigr),\\
B(i,j)&\longrightarrow
 \forall j'<N_H\bigl(E_H(j,j')\to
   \exists i'<N_G(E_G(i,i')\wedge B(i',j'))\bigr).
\end{align*}
The free variables \(G,H,B,i,j\) have just been specified. The two formulas say that the immediate members of the related vertices can be matched in both directions. If such a relation \(B\) exists and \(B(r_G,r_H)\), write \(G\simeq H\). We take all rooted graphs as the interpreting domain, use \(\simeq\) as equality, and define
\[
 G\in^*H\quad\longleftrightarrow\quad
 \exists j<N_H\bigl(E_H(r_H,j)\wedge G\simeq H[j]\bigr),
\]
where \(H[j]\) is the same graph with \(j\) chosen as its root. Thus \(G\) represents the same set as one of the immediate members of \(H\). Composition of bisimulations shows that \(\simeq\) is an equivalence relation and that \(\in^*\) respects it. Extensionality follows directly from the two back-and-forth conditions above.

Empty Set, Pairing, Union, and Set Difference are obtained by taking disjoint copies of finitely many input graphs and adjoining a new root. Kuratowski ordered pairs and Cartesian products require only finitely many repetitions of the same construction. For each fixed \(\Delta_0\) formula, truth is evaluated recursively by bounded searches through the immediate members of the relevant roots; this gives \(\Delta_0\)-Separation. Given two input sets \(x,y\), the operation \(R_8(x,y)\) is obtained by adding, for every immediate member \(w\) of \(y\), a new vertex whose immediate members are precisely the \(u\) satisfying \(\langle u,w\rangle\in x\). The number of vertices required by each of these constructions is bounded by a fixed polynomial in the sizes of the input graphs.

One can likewise give finite tables listing the height of each vertex and all vertices reachable from the root. A vertex of least height gives Set Foundation. Adding one more root above all reachable vertices gives the least transitive set containing the input. This interprets \(\Tzero\) in \(\IDelta+\OmegaOne\). For a more detailed version of this finite-graph interpretation, compare \citet{Chen2026}.

\begin{theorem}
  \(\Q\) interprets \(\Tzero+\Coll{\Sigma_1}\).
\end{theorem}
\begin{proof}
  By \eqref{eq:q-spine}, it is enough to extend the preceding interpretation in \(\BS1+\OmegaOne\) so that it satisfies \(\Sigma_1\)-Collection. Fix a \(\Sigma_1\) formula \(\varphi(x,y,\bar p)\), written as
  \[
    \varphi(x,y,\bar p)\equiv\exists z\,\delta(x,y,z,\bar p),
  \]
  where \(x\) denotes a member of the input set, \(y,z\) are the two sets to be found, \(\bar p\) is a tuple of parameters, and \(\delta\) is \(\Delta_0\). Let \(G\) represent the set \(a\) in the Collection instance, let \(i<N_G\) denote one of its immediate members, let \(Y,Z\) be two candidate graphs, and let \(\bar P\) represent the parameters \(\bar p\). Consider the arithmetical formula
  \[
  W(i,Y,Z,\bar P)\longleftrightarrow
  Y,Z\text{ are finite directed acyclic graphs and }
  \delta(G[i],Y,Z,\bar P)\text{ holds}.
  \]
  Every bounded quantifier in \(\delta\) is translated into a search through finitely many immediate members. Apart from the initial variables \(Y,Z\), all quantifiers on the right are therefore bounded, so \(W\) is a \(\Sigma_1^0\) formula.

  Suppose the translation of the antecedent of Collection holds. The principle \(B\Sigma_1\) gives a number \(b\) such that, for every relevant \(i<N_G\), some pair \(Y,Z<b\) satisfies \(W(i,Y,Z,\bar P)\). For each \(i\), choose the lexicographically least such pair. Place disjoint copies of the resulting graphs \(Y\) on separate vertex intervals, and add a new root pointing to their old roots. If \(Y<b\), the number of vertices actually used by \(Y\) is bounded linearly by the binary length of \(b\). The total number of vertices in the new graph is consequently bounded by a fixed polynomial in \(N_G\) and that length. The axiom \(\Omega_1\) is therefore sufficient to represent all its edges. For each member of \(G\), the set represented by the new root contains a \(\varphi\)-witness. This proves the given instance of \(\Sigma_1\)-Collection.
\end{proof}

\begin{theorem}
  \(\Q\) interprets \(\Tzero+\Inf\).
\end{theorem}
\begin{proof}
  By \eqref{eq:q-spine}, we work in \(\IDelta+\OmegaOne\). A finite graph cannot contain a set whose members are exactly all finite ordinals, so we instead use terms arranged in levels. The construction is modelled on the \(G_2\) construction of \citet{Mathias2006}. Let \(W_0=\N\), and fix a natural number \(t_0\) depending only on the finitely many elementary set operations required by \(\GJ\). For every natural number \(r\), define recursively
  \[
    W_{r+1}=\{\langle t,u,v\rangle:t<t_0,\ u,v\in W_r\}.
  \]
  Here \(r\) is the level, \(t\) specifies the operation, and \(u,v\) are its two inputs; the triple is represented by the usual arithmetical pairing functions.

  At level zero, put \(u\in_0v\) exactly when \(u<v\), and put \(u=_0v\) exactly when \(u=v\). For each elementary set operation there is a fixed definition reducing equality and membership between two terms at level \(r+1\) to finitely many equalities and membership statements at level \(r\). Given \(r\) and finitely many terms, list together all numerical sequences and truth tables needed for the successive calculations. A bounded formula checks that the first row is given by \(<\) and \(=\), and that each later row follows from its predecessor by the fixed definitions. The finite geometric progressions and products needed to place every row in a single number are included in the same table. The theory \(\IDelta+\OmegaOne\) verifies every correct table and proves that every correct table through level \(r\) can be extended by the next row.

  An object in the interpretation is a triple \((r,w,Z)\), where \(w\in W_r\) and \(Z\) is a complete finite table from level zero through level \(r\). When two objects lie at different levels, repeatedly regard the lower-level term as the corresponding term one level higher, and then read equality or membership from the table at the higher level. The fixed definitions immediately give
  \[
    u=_rv\longleftrightarrow ju=_{r+1}jv,
    \qquad
    u\in_rv\longleftrightarrow ju\in_{r+1}jv,
  \]
  where \(j\) is the fixed operation that places an old term at the next level without changing the set it represents. Hence equality is an equivalence relation, membership respects it, and both are defined without parameters.

  Empty Set, Pairing, Set Difference, Union, Cartesian Product, and \(R_8\) are given by their corresponding fixed terms; applying one operation adds only a fixed finite number of rows. For any fixed \(\Delta_0\) formula, replace its atomic formulas by the equality and membership relations just defined, and then proceed through its Boolean connectives and bounded quantifiers. Its truth value is thereby determined by a finite table, which also yields \(\Delta_0\)-Separation.

  The first level of an object \((r,w,Z)\) is the least level at most \(r\) containing a term equal to \(w\). If this number is positive, every member of the object can be represented at a lower level; if it is zero, membership is simply \(<\) on the natural numbers. Choosing first the least level and, at level zero, the least natural number gives Set Foundation.

  One fixed elementary operation forms a set \(\Omega\) whose members are exactly all level-zero terms. These terms represent the finite von Neumann ordinals \(0,1,2,\ldots\), so \(\Omega\) satisfies \(\Inf\). Finally, for a term \(x\) at level \(r\), define successively
  \[
    C_0(x)=x,
    \qquad C_{k+1}(x)=C_k(x)\cup\bigcup C_k(x)\quad(k<r).
  \]
  The variable \(k\) records the number of completed steps, and \(C_k(x)\) is the set obtained at step \(k\). Each step uses the fixed elementary operations already available, so the whole finite calculation can be written beside \(Z\) and checked row by row. Every member of a set at level \(r+1\) can be represented at level \(r\); hence \(C_r(x)\) is transitive. Conversely, any transitive set containing \(x\) must contain each \(C_k(x)\) in turn. Thus \(C_r(x)\) is the least transitive set containing \(x\), and the interpretation satisfies \(\Tzero+\Inf\).
\end{proof}

The preceding interpretation of \(\PAminus+\mathsf{Coll}\) in \(\Q\) has a set-theoretic analogue.
\begin{theorem}
  \(\Q\) interprets \(\mathsf S_0+\mathsf{Coll}\).
\end{theorem}
\begin{proof}
  By \eqref{eq:q-spine}, we again work in \(\IDelta+\OmegaOne\). We interpret a structure satisfying \(\Szero\) and containing a universal set; every instance of Collection then follows immediately from that universal set.

  Begin again with rooted finite directed acyclic graphs, but attach a positive or negative sign to each root. If \(A\) is a finite list of rooted graphs, let the members of \(+A\) be exactly the objects bisimilar to an entry of \(A\), and let the members of \(-A\) be exactly the objects not bisimilar to any entry of \(A\):
  \[
  X\in^*(+A)\longleftrightarrow\exists a\in A\,(X\simeq a),
  \qquad
  X\in^*(-A)\longleftrightarrow\neg\exists a\in A\,(X\simeq a).
  \]
  The free variable \(X\) denotes an arbitrary object and \(A\) denotes the given finite list. Thus a positive list names the members, whereas a negative list names the nonmembers. For every finite list, one can choose a positive one-element chain whose height is greater than the height of every positive graph in the list. This new object is equal neither to any positive object in the list nor to any negative object in it. It follows that a positive object and a negative object cannot have exactly the same members. Together with the usual finite bisimulation argument, this proves Extensionality.

  The object \(+\varnothing\) is empty, \(-\varnothing\) is universal, and Pairing is obtained with a positive sign. Set Difference is given by
  \[
  \begin{array}{rclcrcl}
  (+A)\setminus(+B)&=&+(A\setminus B),&&
  (+A)\setminus(-B)&=&+(A\cap B),\\
  (-A)\setminus(+B)&=&-(A\cup B),&&
  (-A)\setminus(-B)&=&+(B\setminus A).
  \end{array}
  \]
  Here \(A\) and \(B\) are finite lists, and the unions, intersections, and differences on the right are finite operations. Suppose a positive outer set has positive members \(+A_i\) and negative members \(-B_j\). If there are no negative members, its union is \(+\bigcup_i A_i\). Otherwise its union is
  \[
    -\left(\bigcap_jB_j\setminus\bigcup_iA_i\right).
  \]
  If the outer set is negative, then for any object \(x\) one may choose a new object \(t\) outside its finite exclusion list. The pair \(\{x,t\}\) still belongs to the negative outer set, so its union is the universal set. Each of these operations increases the number of vertices by at most a fixed polynomial in the input sizes, and can therefore be carried out in \(\IDelta+\OmegaOne\). Finally, whenever the antecedent of an instance of Collection holds, the universal set contains every possible witness.
\end{proof}

We next use the additional strength of \(\BS1+\expax\).

\begin{theorem}
  \(\BS1+\expax\) interprets \(\Tzero+\Pow+\Coll{\Sigma_1}\).
\end{theorem}
\begin{proof}
  Use the finite-graph interpretation from the beginning of this section. The proof of \(\Sigma_1\)-Collection is exactly the proof of the first theorem, since the required principle \(B\Sigma_1\) is already part of the present theory. It remains only to verify Power Set.

  Suppose the root of \(G\) has \(s\) distinct immediate members, where members identified by bisimulation are counted only once. For each \(e<2^s\), use the \(i\)-th binary digit of \(e\) to decide whether to include the \(i\)-th member. Take disjoint copies of the chosen graphs and add a new root above them. Add one further root whose immediate members are the \(2^s\) roots just constructed. Every \(e\) yields a subset of \(G\); conversely, every subset of \(G\) is uniquely determined by the members it contains and therefore arises from some \(e<2^s\). The totality of exponentiation guarantees the existence of the new graph, whose outermost root represents the power set of \(G\).
\end{proof}

We now pass to \(\IS{n}\), where \(n\geq1\). The same finite-graph interpretation can then handle formulas of higher complexity.

\begin{theorem}
  For every \(n\geq1\), \(\IS{n}\) interprets
  \[
    \Tzero+\Pow+\Coll{\Sigma_{n+1}}+\Fnd{\Sigma_n}+\Fnd{\Pi_n}.
  \]
\end{theorem}
\begin{proof}
  Continue to use finite directed acyclic graphs. An unbounded set-theoretic quantifier becomes an unbounded quantifier over finite graphs, whereas a bounded quantifier searches only through the finitely many immediate members of a root. Thus, if \(\varphi(x,y,\bar p)\) is \(\Sigma_{n+1}\), the assertion that \(Y\) represents a set satisfying \(\varphi(G[i],Y,\bar P)\) is still \(\Sigma_{n+1}^0\). Here \(G,i,Y,\bar P\) have the same meanings as in the first proof of this section. By \eqref{eq:ind-spine}, \(I\Sigma_n\) interprets \(B\Sigma_{n+1}\). The latter places all the finitely many graphs \(Y\) below one numerical bound. Choose the least such \(Y\) for each \(i\), take disjoint copies, and add a new root; the resulting set witnesses Collection. Since \(I\Sigma_n\) also proves the totality of exponentiation, the preceding construction of Power Set applies unchanged.

  Now fix a \(\Sigma_n\) or \(\Pi_n\) formula \(A(x,\bar p)\), where \(x\) denotes the set to be chosen and \(\bar p\) is a tuple of parameters. For every vertex \(i\) reachable from the root of the input graph \(G\), let \(h(i)\) be the greatest length of a downward path beginning at \(i\). Since the graph is finite and every edge points to a smaller vertex number, \(h\) is obtained by finite recursion. The assertion that there is a reachable vertex of height \(k\) satisfying \(A\) has the same \(\Sigma_n^0\) or \(\Pi_n^0\) complexity as \(A\); here \(k\) is the height, while all quantifiers over vertices and path lengths are bounded. The theory \(I\Sigma_n\) gives the least \(k\) for which the assertion holds, after which we choose the least numbered suitable vertex \(i\) at that height. Every member of \(i\) has height below \(k\), so none of them satisfies \(A\). This proves both \(\Sigma_n\)-Foundation and \(\Pi_n\)-Foundation.
\end{proof}

\begin{theorem}
  For every \(n\geq1\), \(\IS{n}\) interprets \(\Tzero+\Repl{\Sigma_{n+1}}\).
\end{theorem}
\begin{proof}
  Use the finite-graph interpretation from the preceding proof. Fix a \(\Sigma_{n+1}\) formula \(\varphi(x,y,\bar p)\), and suppose that it defines a single-valued total relation on a set \(a\). Let \(G\) represent \(a\), let \(i<N_G\) denote one of its immediate members, let \(Y\) be a candidate value, and let \(\bar P\) represent the parameters. The formula
  \[
    W(i,Y,\bar P)\longleftrightarrow
    Y\text{ is a finite directed acyclic graph and }
    \varphi(G[i],Y,\bar P)
  \]
  is \(\Sigma_{n+1}^0\), since the initial variable \(Y\) can be absorbed into the first existential block of a \(\Sigma_{n+1}^0\) formula. The principle \(B\Sigma_{n+1}\) provides a common bound for all relevant \(Y\). For each \(i\), choose the least \(Y\) below that bound, take disjoint copies of the resulting graphs, and add a new root. If \(Y'\) is another value for the same \(i\), single-valuedness gives \(Y\simeq Y'\). The members of the new root are therefore exactly the values of \(\varphi\) on \(a\), with no extraneous elements.
\end{proof}

\begin{theorem}
  For every \(n\geq1\), \(\IS{n}\) interprets
  \(\Tzero+\Inf+\Fnd{\Sigma_n}\) and
  \(\DB+\Inf+\Fnd{\Pi_n}\).
\end{theorem}
\begin{proof}
  We first treat \(\Sigma_n\)-Foundation. In \(I\Sigma_n\), the level-by-level definition of terms beginning with \(W_0=\omega\) can be carried out simultaneously through all internally finite levels; the finite calculation need no longer be attached separately to each object. Equality and membership are still given by the fixed definitions between adjacent levels, so Empty Set, Pairing, Set Difference, Union, Cartesian Product, and \(R_8\) all exist. If \(x\) first appears at level \(r\), applying \(z\mapsto z\cup\bigcup z\) exactly \(r\) times gives the least transitive set containing \(x\). The structure therefore satisfies \(\Tzero+\Inf\).

  Let \(A(x,\bar p)\) be \(\Sigma_n\). The assertion that level \(r\) contains a term satisfying \(A\) is still \(\Sigma_n^0\); here \(r\) is the level, \(x\) is the set represented by a term, and \(\bar p\) is the tuple of parameters. Choose the least such \(r\), and then the least suitable term at that level. Every member of the chosen set can be represented at a lower level, so the chosen set is \(A\)-minimal.

  For \(\Pi_n\)-Foundation, we omit \(R_8\) and retain only Empty Set, Pairing, Set Difference, Union, and Cartesian Product. Each set is represented by a finite expression. Besides its operation symbols and parameters introduced earlier, an expression lists only finitely many natural-number positions that require separate treatment; membership at every other position is determined by the fixed definition of the operation. These finite lists may be arranged so that every member of a set either comes from an earlier expression or has a smaller natural-number position. Membership therefore strictly decreases a pair of natural numbers, ordered first by the stage at which the expression was formed and then by its natural-number position.

  For any fixed \(\Pi_n\) formula \(A(x,\bar p)\), recursion on the formula shows that its arithmetical translation remains \(\Pi_n^0\). Bounded set quantifiers inspect only finite lists and fixed definitions, and hence add no alternation of unbounded quantifiers. If some set satisfies \(A\), the theory \(I\Sigma_n\) first chooses the earliest expression representing such a set and then the least natural-number position within it. Every member of the resulting set decreases the ordered pair just described, so no member still satisfies \(A\). The zeroth level again supplies \(\omega\), and the structure satisfies \(\DB+\Inf+\Fnd{\Pi_n}\).
\end{proof}

There is a one-level shift for \(\Tzero+\Inf+\Fnd{\Pi_n}\).
\begin{theorem}
  For every \(n\geq1\), \(\IS{n+1}\) interprets \(\Tzero+\Inf+\Fnd{\Pi_n}\).
\end{theorem}
\begin{proof}
  Use the level-by-level term structure over \(\omega\) from the preceding proof. The verification of \(\Tzero+\Inf\) is unchanged. Fix a \(\Pi_n\) formula \(A(x,\bar p)\), where \(x\) is a set variable and \(\bar p\) is a tuple of parameters. To ask whether a level contains a set satisfying \(A\), one must first existentially quantify over a term at that level. Consequently
  \[
    \exists x\bigl(x\text{ appears at level }r\text{ and }A(x,\bar p)\bigr)
  \]
  is a \(\Sigma_{n+1}^0\) formula whose only free variable \(r\) denotes the level. The theory \(I\Sigma_{n+1}\) chooses the least \(r\) for which it holds, and then a term at that level satisfying \(A\). Every member of that term appears at a lower level. If one of those members also satisfied \(A\), the minimality of \(r\) would be contradicted. The chosen term is therefore \(A\)-minimal.
\end{proof}

Realising the same constructions in \(\PA\) gives the uniform endpoint.
\begin{theorem}
  \(\PA\) interprets the following theories:
  \begin{enumerate}
    \item \(\Tzero+\Pow+\mathsf{Coll}+\mathsf{Foundation}\);
    \item \(\Tzero+\Inf+\mathsf{Foundation}\).
  \end{enumerate}
\end{theorem}
\begin{proof}
  For the first theory, use finite directed acyclic graphs. The theory \(\PA\) proves that exponentiation is total, so Power Set exists. The arithmetical translation of any fixed Collection formula has only finitely many quantifier alternations. The corresponding bounding principle in \(\PA\) places all finitely many witnesses below one number, after which their graphs can be copied and placed below a new root as before. Likewise, the translation of any fixed Foundation formula has finite complexity. Among the reachable vertices satisfying it, choose first one of least height and then one of least number. The same interpreting formulas therefore satisfy every instance of Collection and Foundation.

  For the second theory, use the level-by-level term structure over \(\omega\). For any fixed formula, \(\PA\) can carry out its truth definition and choose the first level containing a set that satisfies it. All members of the chosen set occur at lower levels, so every instance of Foundation holds. The successive-union construction above gives the least transitive set containing any given set, and the set \(\Omega\) above level zero gives Strong Infinity.
\end{proof}

\section{Interpreting Arithmetic in Set Theory}

We now turn to the converse direction. Every set theory considered in this section can interpret adjunctive set theory with Extensionality, and hence can interpret \(\Q\); see \citep{CollinsHalpern1970,Damnjanovic2017,Visser2008}. The real work is to obtain exponentiation, induction, and bounding.

\begin{theorem}
  \(\ReSzero+\Pow\) interprets \(\BS1+\expax\).
\end{theorem}
\begin{proof}
  We first show that Cartesian products exist in the present theory. Given sets \(a,b\), put \(c=a\cup b\). Every Kuratowski pair \(\langle x,y\rangle\), with \(x\in a\) and \(y\in b\), belongs to \(\mathcal P(\mathcal P(c))\). Hence \(\Delta_0\)-Separation gives
  \[
  a\times b=\{z\in\mathcal P(\mathcal P(c)):
  \exists x\in a\,\exists y\in b\,z=\langle x,y\rangle\}.
  \]
  Here \(z\) is a candidate ordered pair, and \(x,y\) are its two components. Every quantifier in the displayed formula is bounded by a set already given.

  An interpreted natural number is a pair \((A,L)\), where \(L\subseteq A\times A\) linearly orders \(A\), and every nonempty subset of \(A\) has both an \(L\)-least and an \(L\)-greatest element. Two pairs \((A,L)\) and \((B,K)\) are equal in the interpretation if there is a bijection \(f:A\to B\); the function \(f\) is selected from \(\mathcal P(A\times B)\). Similarly, \((A,L)\leq(B,K)\) means that there is an injection from \(A\) into \(B\). These definitions depend only on the sizes of \(A\) and \(B\), and are therefore well defined on the quotient by bijection.

  We first verify closure under the operations used below. A subset inherits the original order. A tagged disjoint union is ordered block by block, and a Cartesian product lexicographically. On \(\mathcal P(A)\), define, for subsets \(X,Y\subseteq A\),
  \[
  X<_{L}^{\rm lex}Y
  \quad\longleftrightarrow\quad
  X\neq Y\ \text{ and }\ \max_L(X\mathbin\triangle Y)\in Y.
  \]
  The free variables \(X,Y\) are subsets of \(A\), and \(X\mathbin\triangle Y\) is their symmetric difference. Thus, at the last position where the two subsets differ, \(Y\) has value one. Since every nonempty subset of \(A\) has an \(L\)-greatest element, this defines a linear order.

  Let \(C\) be a nonempty subset of \(\mathcal P(A)\). Starting from the greatest position of \(A\) and moving backwards, retain the members of \(C\) that have value zero at the current position whenever any remain; otherwise retain those with value one. Power Set provides the set of all partial functions arising in this procedure, and the union of compatible partial functions completes the successive choices. The unique subset that remains is the least member of \(C\). Reversing the choices of zero and one gives its greatest member. Thus \(\mathcal P(A)\) has the same required property. In particular, for such \(A,B\), the function set
  \[
  A^B=\{f\in\mathcal P(B\times A):f\text{ is a function from }B\text{ to }A\}
  \]
  exists by \(\Delta_0\)-Separation and inherits the preceding order.

  Next, every such \(A\) is Dedekind finite. Suppose that \(f:A\to A\) is injective but not surjective. Among the closed initial segments
  \(I_a=\{x\in A:x\leq_La\}\) that admit a nonsurjective self-injection, choose one whose endpoint \(a\) is least. If \(f(a)=a\), restrict \(f\) to the preceding initial segment. If \(a\) is not in the range of \(f\), remove \(a\) from the domain. In the remaining case, choose \(p\neq a\) with \(f(p)=a\), and replace the arrow \(p\mapsto a\) by \(p\mapsto f(a)\) before removing \(a\). In every case one obtains a nonsurjective self-injection of a proper initial segment, contradicting the minimality of \(a\).

  Any two such sets can also be compared by injections. Inside \(\mathcal P(A\times B)\), take all functions that are order isomorphisms between an initial segment of \(A\) and an initial segment of \(B\). Any two of them agree on their common domain; otherwise the first point of disagreement gives an immediate contradiction. Their union is again an initial-segment isomorphism. If neither side is exhausted, extend it by pairing the least unused elements. Thus at least one side is exhausted, yielding an injection \(A\to B\) or \(B\to A\). If injections exist in both directions, their composite is a self-injection of \(A\). Dedekind finiteness makes the composite surjective, and the original injections are consequently bijections.

  Interpret zero by the size of the empty set, successor by tagged disjoint union with a singleton, and define
  \[
  \begin{aligned}
  |A|+|B|&=|(A\times\{0\})\cup(B\times\{1\})|,\\
  |A|\cdot|B|&=|A\times B|,\\
  2^{|B|}&=|\{0,1\}^{B}|.
  \end{aligned}
  \]
  The variables \(A,B\) denote ordered sets of the kind just described. The three equations use, respectively, tagged disjoint union, Cartesian product, and the set of functions from \(B\) to \(\{0,1\}\). The usual coordinate permutations and associativity maps prove the semiring axioms. Comparability gives a linear order, and adjoining one point outside the range of an injection gives discreteness. The quotient structure is therefore a discretely ordered nonnegative semiring, and exponentiation satisfies its recursive equations.

  It remains to prove bounded induction. Let \(\theta\) be any formula invariant under bijections, and let \((B,K)\) represent its numerical bound. Then
  \[
  \exists [X]<[B]\,\theta([X])
  \quad\longleftrightarrow\quad
  \exists X\in\mathcal P(B)\,
  \bigl(X\subsetneq B\wedge\theta([X,K\!\upharpoonright X])\bigr).
  \]
  On the left, the free variable \([X]\) denotes an interpreted natural number below \([B]\); on the right, \(X\) is a proper subset of \(B\). From left to right take the range of an injection, and from right to left take the inclusion map. Every bounded arithmetical quantifier is therefore translated into a set-theoretic quantifier bounded by a power set.

  If an instance of bounded induction failed at \((A,L)\), use Separation on \(A\) to form the set of endpoints at which the corresponding closed initial segment fails the induction conclusion, and choose its \(L\)-least member \(a\). It cannot be the first element. If \(p\) is its immediate predecessor, then the size of \(I_a\) is the successor of the size of \(I_p\), and the induction hypothesis together with the minimality of \(a\) gives a contradiction. Thus the interpreted structure satisfies \(I\Delta_0+\expax\), and by \eqref{eq:exp-spine} it interprets \(\BS1+\expax\).
\end{proof}

We next interpret \(\IS{n}\) in set theory. We first fix the natural-number domain. For sets \(a,b\), let
\[
\operatorname{Succ}(a,b)\longleftrightarrow
a\in b\wedge\forall z\in b\,(z\in a\vee z=a)
\wedge\forall z\in a\,(z\in b).
\]
The variable \(a\) denotes the predecessor and \(b\) the successor; the formula says exactly that \(b=a\cup\{a\}\). Let \(N(a)\) assert that \(a\) is empty, or that \(a\) is transitive, linearly ordered by membership, and that \(a\) and each of its nonempty members have an immediate predecessor satisfying the displayed formula. This definition contains only bounded quantifiers. Pairing and Union make \(a\mapsto a\cup\{a\}\) total. The relevant Foundation scheme makes predecessors unique and ensures that any two sets satisfying \(N\) are comparable by membership.

Addition and multiplication need not themselves exist as sets at the outset. Given \(a,b,c\) satisfying \(N\), say that a set \(f\) is an addition table beginning at \(a\), of length \(b\), and ending at \(c\) if \(f\) is single-valued on \(b\cup\{b\}\), \(f(0)=a\), \(f(b)=c\), and consecutive values are related by \(\operatorname{Succ}\). Once \(f\) is given, every quantifier is bounded by \(f\), \(b\), or \(\bigcup\bigcup f\), so the condition is \(\Delta_0\). A multiplication table begins with zero and, between consecutive values, includes an addition table showing that the later value is obtained from the earlier one by adding \(a\). This condition also contains only bounded quantifiers.

\begin{theorem}
  For every \(n\geq1\), both \(\ReSzero+\Fnd{\Sigma_n}\) and
  \(\ReSzero+\Fnd{\Pi_n}\) interpret \(\IS{n}\).
\end{theorem}
\begin{proof}
  Assume first \(\Pi_n\)-Foundation. Fix \(a\) satisfying \(N\). If some \(b\) has no complete addition table beginning at \(a\), then the assertion that \(b\) has this property is \(\Pi_1\), and hence \(\Pi_n\). Choose such a \(b\) that is minimal under membership. It is not zero. Let \(b^-\) be its predecessor. By minimality there is an addition table through \(b^-\); if its last value is \(c\), adjoining the pair \(\langle b,c\cup\{c\}\rangle\) extends it through \(b\), a contradiction. Addition is therefore total.

  It is also single-valued. If two addition tables disagree, \(\Delta_0\)-Separation gives the set of positions at which their values differ. Choose its least member. The two tables agree at the preceding position, and uniqueness of successors then forces them to agree at the chosen position as well. The same arguments prove that multiplication is total and single-valued.

  For each atomic arithmetical equation, both truth and falsity now have \(\Sigma_1\) definitions: truth is witnessed by a table with the stated output, while falsity is witnessed by a table with a different output. After terms are replaced by their graph relations and formulas are put in negation normal form, recursion on formulas shows that \(\Sigma_k^0\) and \(\Pi_k^0\) formulas translate, respectively, into \(\Sigma_k\) and \(\Pi_k\) formulas. If an instance of \(\Sigma_n^0\)-induction failed, its counterexamples would therefore form a \(\Pi_n\) class. A membership-minimal counterexample is neither zero nor a successor: the base clause excludes zero, and at a successor the minimality assumption and the induction step give a contradiction. Thus \(\ReSzero+\Fnd{\Pi_n}\) interprets \(I\Sigma_n\).

  Now assume \(\Sigma_n\)-Foundation. When \(n\geq2\), the preceding construction of addition and multiplication still applies, since its \(\Pi_1\) failure predicates can be placed in the strict class \(\Sigma_n\) by adding vacuous quantifier blocks. A failed instance of \(\Pi_n^0\)-induction has a \(\Sigma_n\) class of counterexamples. The same minimal-counterexample argument proves \(I\Pi_n\), which is equivalent to \(I\Sigma_n\).

  It remains to treat \(\Sigma_1\)-Foundation. Let
  \(\operatorname{Step}(u,v,e,\bar p)\) be a \(\Delta_0\) formula, where \(u\) is the current value, \(v\) is the next value when the calculation is read backwards, \(e\) is an auxiliary set used in checking the step, and \(\bar p\) is a tuple of parameters. Assume that for every admissible \(u\) there are \(v,e\) satisfying \(\operatorname{Step}(u,v,e,\bar p)\), and that \(v\) is unique.

  For \(c\in b\cup\{b\}\), let \(R(c,b,a,z,h,\bar p)\) assert that \(h\) lists all values from position \(c\) through position \(b\), that the value at \(b\) is \(a\), that the value at \(c\) is \(z\), and that each pair of adjacent rows satisfies \(\operatorname{Step}\). The free variable \(b\) is the total length, \(a\) is the prescribed value at the upper endpoint, \(z\) is the value at position \(c\), and \(h\) is the whole finite table. Once \(h\) is given, every quantifier is bounded, so \(R\) is \(\Delta_0\). Consequently
  \[
  S(c)\longleftrightarrow
  c\in b\cup\{b\}\wedge\exists z\,\exists h\,R(c,b,a,z,h,\bar p)
  \]
  is \(\Sigma_1\). Its only free variable \(c\) records how far the table has been completed backwards from its upper endpoint.

  Clearly \(S(b)\) holds. Choose a membership-minimal \(c\) satisfying \(S\). If \(c\) is the successor of \(d\), totality of \(\operatorname{Step}\) extends the existing table by one row at its beginning, giving \(S(d)\), a contradiction. Hence \(c=0\), and every internally finite calculation of this kind can be completed. Its value at zero is unique as well. Given two complete tables, use \(\Delta_0\)-Separation to form the set of positions at which their values agree. This set contains \(b\). If its membership-minimal element were a successor, functionality of \(\operatorname{Step}\) would put its predecessor in the same set. Its minimal element is therefore zero.

  Taking \(\operatorname{Step}(u,v,e)\) to mean that \(v\) is the successor of \(u\) gives addition. For multiplication, let the auxiliary set \(e\) include a complete addition table showing \(v=u+a\). For any fixed arithmetical \(\Delta_0^0\) formula, proceed recursively through the formula: evaluate a bounded existential quantifier by taking successive disjunctions while moving from its bound down to zero, and evaluate a bounded universal quantifier by successive conjunctions. The backward calculation just proved supplies a unique result in both cases.

  It follows that a \(\Pi_1^0\) formula translates into a \(\Pi_1\) formula, while its negation translates into a \(\Sigma_1\) formula. If an instance of \(\Pi_1^0\)-induction failed, \(\Sigma_1\)-Foundation would give a membership-minimal counterexample. The base and successor clauses again exclude the two possible cases. Thus the interpreted arithmetic satisfies \(I\Pi_1=I\Sigma_1\).
\end{proof}

\begin{theorem}
  For every \(n\geq1\), \(\ReSzero+\Sep{\Sigma_n}\) interprets \(\IS{n}\).
\end{theorem}
\begin{proof}
  Throughout the proof write \(a^+=a\cup\{a\}\). We first record a finite recursion fact that will be used repeatedly. Let \((A,L)\) linearly order \(A\), and suppose that every nonempty subset of \(A\) has both an \(L\)-least and an \(L\)-greatest element. Let \(\delta(x,y,e,\bar p)\) be a \(\Delta_0\) formula such that, for every \(x\in A\), there exists a unique \(y\) for which some \(e\) satisfies \(\delta(x,y,e,\bar p)\). Then the graph of this function on \(A\) is a set.

  For \(a\in A\), let \(P(a)\) assert that there are two functions \(f,g\), both defined on the closed initial segment \(A_{\leq a}\), such that
  \[
    \delta(x,f(x),g(x),\bar p)
    \qquad(x\in A_{\leq a}).
  \]
  The formula \(P\) is \(\Sigma_1\). First use \(\Sigma_1\)-Separation to collect the elements satisfying \(P\), and then take the difference to obtain the set of elements failing \(P\). If this set is nonempty, take its \(L\)-least element \(a\). If \(a\) is the first element, Pairing gives the required one-point functions. Otherwise \(A_{<a}\) has an \(L\)-greatest element, and two functions defined on that proper initial segment need only be extended by one ordered pair each to cover \(a\). Both cases contradict minimality. Hence \(P(a)\) holds for every \(a\in A\), and taking the last element of \(A\) gives the full functions. Applying this argument twice also yields Cartesian products of two such sets with their lexicographic orders, as well as images of functions on such sets. This conclusion applies only to a given ordered set \(A\), and does not imply Cartesian products, \(R_8\), or Replacement for arbitrary sets.

  Now define the interpreted natural numbers. Let \(\operatorname{Chain}(a)\) be the following \(\Delta_0\) condition: \(a\) is transitive, membership linearly orders \(a\), and every nonzero element of \(a^+\) is a successor. Following \citet[Section 10]{Mathias2006}, use the \(\Sigma_1\) formula
  \[
    \Seq(S,2,a^+),
  \]
  which asserts that \(S\) contains every binary sequence whose domain is an initial segment of \(a^+\). Define
  \[
    \operatorname{supp}(f)=\{x\in a:f(x)=1\}.
  \]
  This is obtained by \(\Delta_0\)-Separation on \(a\). Then
  \[
  \begin{aligned}
    \operatorname{Num}(a)\longleftrightarrow\exists S\,[{}&
    \operatorname{Chain}(a)\wedge\Seq(S,2,a^+)\wedge{}\\
    &\forall f\in S\,(f\colon a\to2\longrightarrow\\
    &\qquad[\operatorname{supp}(f)=\varnothing\ \vee\
    \operatorname{supp}(f)
    \text{ has a membership-least and a membership-greatest element}])].
  \end{aligned}
  \]
  The formula \(\operatorname{Num}\) is \(\Sigma_1\), and its negation is \(\Pi_1\).

  The class of sets satisfying \(\operatorname{Num}\) is closed under zero, successor, and members. The zero case is immediate. If \(b\in a\), restrict the witnessing \(S\) to sequences whose domain is an initial segment of \(b^+\). Now suppose \(S\) witnesses \(\operatorname{Num}(a)\), and let
  \[
    F_a=\{f\in S:f\colon a\longrightarrow2\}.
  \]
  Order \(F_a\) by the short lexicographic order, in which two functions are compared at the largest coordinate where they differ. For any nonempty \(C\subseteq F_a\), define nonempty subsets \(C_x\) by recursion along the reverse order of \(a\): at coordinate \(x\), if some function in the current set takes value zero at \(x\), keep exactly those functions; otherwise keep those taking value one. This recursion is carried out by the finite recursion fact above. The single function that remains is the least member of \(C\). Interchanging zero and one gives the greatest member. Thus \(F_a\) again has the property that every nonempty subset has a least and a greatest element.

  Using the preceding Cartesian product and image constructions, form
  \[
    \{f\cup\{\langle a,i\rangle\}:f\in F_a,\ i\in2\}.
  \]
  Combine this set with \(S\) to obtain all binary sequences on \(a^+\). The newly added top coordinate also ensures that every nonempty nonzero set of positions has a least and a greatest element. Therefore
  \[
  \operatorname{Num}(0),\qquad
  \operatorname{Num}(a)\longrightarrow\operatorname{Num}(a^+),\qquad
  \operatorname{Num}(a)\wedge b\in a\longrightarrow\operatorname{Num}(b).
  \]

  Moreover, every set satisfying \(\operatorname{Num}\) has the double well-order property just described. If \(X\subseteq a\) is nonempty, \(\Delta_0\)-Separation produces the characteristic function of \(X\) inside the set supplied by \(S\). This function belongs to \(S\), so \(X\) has a least and a greatest element. If both \(a\) and \(b\) satisfy \(\operatorname{Num}\), a least-counterexample argument in \(a^+\) and \(b^+\) shows that
  \[
    a\in b\quad\text{or}\quad a=b\quad\text{or}\quad b\in a,
  \]
  and exactly one of the three alternatives holds. More precisely, one first shows that the two transitive chains are comparable by inclusion. If the inclusion is proper without equality, the least element of the difference of the longer chain is the shorter chain itself. Thus equality in the interpretation is set equality, and the order is membership.

  Now define addition and multiplication. For \(a,b,c\) satisfying \(\operatorname{Num}\), let \(\operatorname{Add}(a,b,c)\) assert that there is a function \(f\) with domain \(b^+\) such that
  \[
    f(0)=a,\qquad f(b)=c,\qquad
    f(i^+)=f(i)^+\quad(i\in b).
  \]
  In this definition one simultaneously quantifies over a function with domain \(b^+\), together with, for each \(f(i)\), the sets needed to verify \(\operatorname{Num}(f(i))\). After this initial quantifier, only bounded quantifiers remain. Hence \(\operatorname{Add}\) is \(\Sigma_1\).

  For any \(d\in b^+\), let \(P(d)\) assert that there is a function satisfying the recursive equations on domain \(d^+\). The formula \(P\) is \(\Sigma_1\). If some element of \(b^+\) failed \(P\), choose the least such \(d\). It is not zero. If \(d=e^+\), then a function defined on \(e^+\) can be extended to \(d^+\) by adjoining \(\langle d,f(e)^+\rangle\), a contradiction. Addition is therefore total. If two such functions have different values, choose the least position of disagreement in \(b^+\). It is not zero, and the values at the preceding position are equal; uniqueness of successors forces the values at the chosen position to be equal too. Addition is therefore single-valued.

  Multiplication is defined in exactly the same way: set \(f(0)=0\), and require \(f(i^+)\) to be the sum of \(f(i)\) and \(a\). After existentially quantifying over the addition functions needed at each step, multiplication is again given by a \(\Sigma_1\) formula. The same least-counterexample argument proves that it is total and single-valued. Thus we have relations
  \[
    \operatorname{Add}(a,b,c),\qquad
    \operatorname{Mul}(a,b,c),
  \]
  satisfying the usual recursive equations. Since outputs are unique, failure of addition is also expressible by a \(\Sigma_1\) formula:
  \[
    \neg\operatorname{Add}(a,b,c)\longleftrightarrow
    \exists c'\bigl(\operatorname{Num}(c')\wedge
    \operatorname{Add}(a,b,c')\wedge c'\ne c\bigr),
  \]
  and the same holds for multiplication. Together with the linear order obtained above, this shows that all atomic arithmetical formulas and their negations are equivalent, after interpretation, to \(\Sigma_1\) formulas.

  By induction on the structure of a \(\Delta_0^0\) formula, one can further show that \(\theta^I\) and \(\neg\theta^I\) are both equivalent to \(\Sigma_1\) formulas. Boolean connectives present no difficulty. For a bounded universal quantifier, suppose that for every \(x\in a\) there is a witness for the corresponding \(\Sigma_1\) formula. Taking the truth value as the unique output and the witness as an auxiliary set, apply the finite recursion fact at the beginning of the proof to obtain a function defined on all of \(a\). The negation of a bounded existential quantifier is handled similarly. The functions used here are all defined on some set satisfying \(\operatorname{Num}\), so no Collection axiom is needed.

  Now relationalise the arithmetical language and put formulas in negation normal form. Existential quantifiers are restricted by the \(\Sigma_1\) definition of \(\operatorname{Num}\), and universal quantifiers outside the interpreted domain are handled by the \(\Pi_1\) definition of \(\neg\operatorname{Num}\); bounded arithmetical quantifiers are translated directly into bounded membership quantifiers. Using the simultaneous treatment of \(\Delta_0^0\) formulas and their negations, induction on the quantifier prefix shows that, for every \(k\geq1\), arithmetical \(\Sigma_k^0\) formulas translate into \(\Sigma_k\) formulas, and arithmetical \(\Pi_k^0\) formulas translate into \(\Pi_k\) formulas.

  Finally fix a \(\Sigma_n^0\) formula \(\varphi(x,\bar p)\), and assume its zero clause and successor induction step. If the induction conclusion fails, choose \(a\) satisfying \(\operatorname{Num}(a)\) such that \(\neg\varphi^I(a,\bar p)\). By the complexity calculation just given, Separation on \(a^+\) forms
  \[
    C=\{x\in a^+:\varphi^I(x,\bar p)\},
    \qquad F=a^+\setminus C.
  \]
  The set \(F\) is nonempty, and \(a^+\) has the double well-order property, so \(F\) has a least element \(b\). This \(b\) is not zero. If \(b=c^+\), then minimality gives \(\varphi^I(c,\bar p)\), and the successor induction step gives \(\varphi^I(b,\bar p)\), a contradiction. Therefore the interpreted arithmetic satisfies \(I\Sigma_n\). Taking \(n=1\) for induction, and then using the recursive equations for addition and multiplication, proves associativity, commutativity, distributivity, and the remaining axioms of a discretely ordered semiring. This gives the required arithmetical structure.
\end{proof}

The final finite-level phenomenon is that Infinity raises the interpreted induction level.
\begin{theorem}
  For every \(n\geq1\), \(\Tzero+\Inf+\Fnd{\Pi_n}\) interprets \(\IS{n+1}\).
\end{theorem}
\begin{proof}
  On \(\omega\), zero, successor, addition, and multiplication are still defined by the formulas for finite ordinals used above. The scheme \(\Pi_n\)-Foundation proves that these operations are total and single-valued. It is therefore enough to prove the least-number principle for \(\Sigma_{n+1}^0\) formulas.

  Fix a Kuratowski representation of ordered triples. For a finite ordinal \(i\), a natural number \(u\), and a set \(s\), define
  \[
    H(i,u,s)=\bigl\{\{\langle i,u,s\rangle\}\bigr\},
    \qquad C(p,i,u,s)=p\cup\{\varnothing,H(i,u,s)\}.
  \]
  Here \(p\) denotes all elements already formed, \(i\) is the current row number, \(u\) is the natural-number parameter for that row, and \(s\) is the finite function being constructed. The empty set and \(H(i,u,s)\) have different forms inside \(C(p,i,u,s)\), while every member inherited from \(p\) contains the empty set. Thus \(i,u,s\) can be recovered uniquely from \(C(p,i,u,s)\).

  For each finite ordinal \(i\), construct a function \(s\) with domain \(i\cup\{i\}\) and \(s(0)=\varnothing\). For every \(j<i\), let \(q_j=s\upharpoonright(j\cup\{j\})\) be the first \(j+1\) rows, and require
  \[
    s(j+1)=s(j)\cup\{C(s(j),j,u,q_j):u\in\omega\}.
  \]
  In this formula, \(j\) is the current row number and \(q_j\) contains only rows already completed. The set in braces is the image of \(\omega\) under the fixed relation \(u\mapsto C(s(j),j,u,q_j)\), and hence exists by \(R_8\). Once \(s\) is given, all the displayed equalities are checked using only quantifiers bounded by \(s\), \(i\), and \(\omega\). Consequently, the assertion that no such \(s\) exists at \(i\) is \(\Pi_1\). A function exists at zero. If there were a least failed \(i\), its predecessor's function could be extended by the displayed equation, a contradiction. Such an \(s\) therefore exists for every finite ordinal \(i\).

  Now fix an arithmetical \(\Sigma_{n+1}^0\) formula
  \[
    \exists u\,\psi(u,i,\bar a),
  \]
  where \(i\) is the natural number under consideration, \(u\) is a witness, \(\bar a\) is a tuple of parameters, and \(\psi\) is \(\Pi_n^0\). Put \(q_i=s\upharpoonright(i\cup\{i\})\). For a set of the form \(c=C(s(i),i,u,q_i)\), let \(A(c)\) assert that \(\psi(u,i,\bar a)\) holds. The values \(i,u,q_i\) are recovered uniquely from \(c\), using only quantifiers bounded by \(\omega\) and finitely many iterated unions of \(c\). Hence \(A\) is \(\Pi_n\).

  If the original arithmetical formula holds at \(i\), choose a corresponding \(u\); then \(C(s(i),i,u,q_i)\) satisfies \(A\). By \(\Pi_n\)-Foundation, choose a membership-minimal set satisfying \(A\), and let \(i_0\) be the row number recovered from it. If the original formula also held at some \(j<i_0\), the corresponding set from row \(j\) would belong to \(s(i_0)\), hence would be a member of the chosen set, and would also satisfy \(A\). This contradicts minimality. Thus \(i_0\) is the least natural number at which the original formula holds. This is the \(\Sigma_{n+1}^0\) least-number principle, which is equivalent to \(I\Sigma_{n+1}\).
\end{proof}

We finish the reverse direction with the theories that interpret all of \(\PA\).
\begin{theorem}
  Each of the following theories interprets \(\PA\):
  \begin{enumerate}
    \item \(\ReSzero+\mathsf{Foundation}\);
    \item \(\ReSzero+\mathsf{Repl}\).
  \end{enumerate}
\end{theorem}
\begin{proof}
  For the first theory, use the same natural-number domain and the same formulas for addition and multiplication as in the Foundation theorem. Every arithmetical formula \(\varphi(m,\bar p)\) belongs to some finite level of the hierarchy. The corresponding instance of full Foundation gives a membership-minimal counterexample satisfying \(N(m)\wedge\neg\varphi(m,\bar p)\); the base and successor clauses of induction exclude zero and successors. Every arithmetical induction axiom therefore holds.

  For the second theory, use the interpretation from the Replacement theorem. For each \(n\), full Replacement includes \(\Sigma_{n+1}\)-Replacement, so the interpreted arithmetic satisfies \(I\Sigma_n\). Every individual arithmetical induction axiom belongs to some finite fragment \(I\Sigma_n\), and hence all induction axioms hold together.
\end{proof}

\section{Summary and Outlook}

We conclude with diagrams summarising the classification.

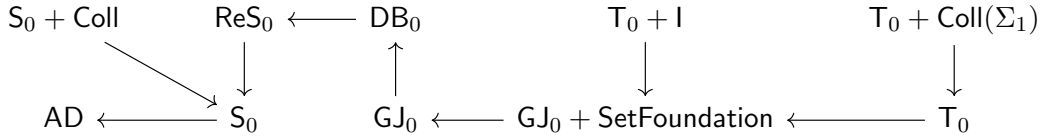
\begin{figure}[htbp]
\centering
\begin{adjustbox}{max width=\textwidth}
\begin{tikzcd}
	{\Szero+\mathsf{Coll}} & \ReSzero & \DB & {\Tzero+\Inf} & {\Tzero+\Coll{\Sigma_1}} \\
	\AD & \Szero & \GJ & {\GJ+\SetFnd} & \Tzero
	\arrow[from=1-1, to=2-2]
	\arrow[from=1-2, to=2-2]
	\arrow[from=1-3, to=1-2]
	\arrow[from=1-4, to=2-4]
	\arrow[from=1-5, to=2-5]
	\arrow[from=2-2, to=2-1]
	\arrow[from=2-3, to=1-3]
	\arrow[from=2-4, to=2-3]
	\arrow[from=2-5, to=2-4]
\end{tikzcd}
\end{adjustbox}
\caption{Set theories mutually interpretable with \(\Q\)}
\end{figure}

\begin{figure}[htbp]
\centering
\begin{adjustbox}{max width=\textwidth}
\begin{tikzcd}
	{\ReSzero+\Pow+\Coll{\Sigma_1}} & {\DB+\Pow+\Coll{\Sigma_1}} & {\GJ+\Pow+\Coll{\Sigma_1}} & {\Tzero+\Pow+\Coll{\Sigma_1}} \\
	{\ReSzero+\Pow} & {\DB+\Pow} & {\GJ+\Pow} & {\Tzero+\Pow}
	\arrow[from=1-1, to=2-1]
	\arrow[from=1-2, to=1-1]
	\arrow[from=1-2, to=2-2]
	\arrow[from=1-3, to=1-2]
	\arrow[from=1-3, to=2-3]
	\arrow[from=1-4, to=1-3]
	\arrow[from=1-4, to=2-4]
	\arrow[from=2-2, to=2-1]
	\arrow[from=2-3, to=2-2]
	\arrow[from=2-4, to=2-3]
\end{tikzcd}
\end{adjustbox}
\caption{Set theories mutually interpretable with \(\BS1+\expax\)}
\end{figure}

\begin{figure}[htbp]
\centering
\begin{adjustbox}{max width=\textwidth}
\begin{tikzcd}
	{\Tzero+\Inf+\mathsf{Foundation}} & {\Tzero+\mathsf{Foundation}} && {\ReSzero+\mathsf{Foundation}} \\
	{\Tzero+\Pow+\mathsf{Coll}+\mathsf{Foundation}} & {\Tzero+\mathsf{Coll}} & {\ReSzero+\mathsf{Coll}} & {\ReSzero+\mathsf{Repl}}
	\arrow[from=1-1, to=1-2]
	\arrow[from=1-2, to=1-4]
	\arrow[from=2-1, to=1-2]
	\arrow[from=2-1, to=2-2]
	\arrow[from=2-2, to=2-3]
	\arrow[from=2-3, to=2-4]
\end{tikzcd}
\end{adjustbox}
\caption{A short list of set theories mutually interpretable with \(\PA\)}
\end{figure}
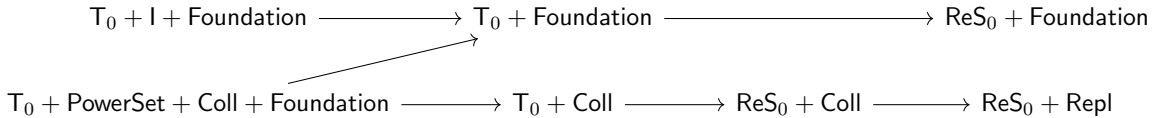

\clearpage
\begin{landscape}
\thispagestyle{plain}
\vspace*{\fill}
\begin{figure}[H]
\centering
\begin{adjustbox}{max width=0.98\linewidth,max totalheight=0.78\textheight,center}
\begin{tikzcd}[column sep=small,row sep=small]
	&& {\Tzero+\Pow+\Coll{\Sigma_{n+1}}+\Fnd{\Sigma_n}+\Fnd{\Pi_n}} &&& \\
	& {\Tzero+\Pow+\Coll{\Sigma_{n+1}}+\Fnd{\Sigma_n}} && {\Tzero+\Pow+\Coll{\Sigma_{n+1}}+\Fnd{\Pi_n}} \\
	{\Tzero+\Pow+\Fnd{\Sigma_{n}}} & {\Tzero+\Coll{\Sigma_{n+1}}+\Fnd{\Sigma_n}} & {\Tzero+\Pow+\Coll{\Sigma_{n+1}}} & {\Tzero+\Coll{\Sigma_{n+1}}+\Fnd{\Pi_n}} && {\Tzero+\Pow+\Fnd{\Pi_{n}}} \\
	& {\Tzero+\Inf+\Fnd{\Sigma_{n}}} & {\Tzero+\Coll{\Sigma_{n+1}}} & {\Tzero+\Inf+\Fnd{\Pi_{n-1}}(n\geq 2)} & {\Tzero+\Fnd{\Pi_{n}}} & {\Tzero+\Repl{\Sigma_{n+1}}} \\
	{\DB+\Pow+\Fnd{\Sigma_{n}}} & {\Tzero+\Fnd{\Sigma_{n}}} & {\DB+\Pow+\Coll{\Sigma_{n+1}}} & {\Tzero+\SColl{\Sigma_{n}}} & {\Tzero+\Sep{\Sigma_n}} & {\DB+\Pow+\Fnd{\Pi_{n}}} \\
	& {\DB+\Inf+\Fnd{\Sigma_{n}}} & {\DB+\Coll{\Sigma_{n+1}}} & {\DB+\Inf+\Fnd{\Pi_{n}}} & {\DB+\Fnd{\Pi_{n}}} & {\DB+\Repl{\Sigma_{n+1}}} \\
	{\ReSzero+\Pow+\Fnd{\Sigma_{n}}} & {\DB+\Fnd{\Sigma_{n}}} & {\ReSzero+\Pow+\Coll{\Sigma_{n+1}}} & {\DB+\SColl{\Sigma_n}} & {\DB+\Sep{\Sigma_n}} & {\ReSzero+\Pow+\Fnd{\Pi_{n}}} \\
	& {\ReSzero+\Inf+\Fnd{\Sigma_{n}}} & {\ReSzero+\Coll{\Sigma_{n+1}}} & {\ReSzero+\Inf+\Fnd{\Pi_{n}}} & {\ReSzero+\Fnd{\Pi_{n}}} & {\ReSzero+\Repl{\Sigma_{n+1}}} \\
	& {\ReSzero+\Fnd{\Sigma_{n}}} && {\ReSzero+\SColl{\Sigma_n}} & {\ReSzero+\Sep{\Sigma_n}}
	\arrow[from=1-3, to=2-2]
	\arrow[from=1-3, to=2-4]
	\arrow[from=2-2, to=3-1]
	\arrow[from=2-2, to=3-2]
	\arrow[from=2-2, to=3-3]
	\arrow[from=2-4, to=3-3]
	\arrow[from=2-4, to=3-4]
	\arrow[from=2-4, to=3-6]
	\arrow[from=3-1, to=5-1]
	\arrow[from=3-1, to=5-2]
	\arrow[from=3-2, to=4-3]
	\arrow[bend left=24, from=3-2, to=5-2]
	\arrow[from=3-3, to=4-3]
	\arrow[bend left=24, from=3-3, to=5-3]
	\arrow[from=3-4, to=4-3]
	\arrow[from=3-4, to=4-5]
	\arrow[from=3-6, to=4-5]
	\arrow[bend left=24, from=3-6, to=5-6]
	\arrow[from=4-2, to=5-2]
	\arrow[bend right=24, from=4-2, to=6-2]
	\arrow[from=4-3, to=5-4]
	\arrow[bend right=24, from=4-3, to=6-3]
	\arrow[bend right=24, from=4-5, to=6-5]
	\arrow[from=4-6, to=5-4]
	\arrow[bend right=24, from=4-6, to=6-6]
	\arrow[from=5-1, to=7-1]
	\arrow[from=5-1, to=7-2]
	\arrow[bend left=24, from=5-2, to=7-2]
	\arrow[from=5-3, to=6-3]
	\arrow[bend left=24, from=5-3, to=7-3]
	\arrow[from=5-4, to=5-5]
	\arrow[bend left=24, from=5-4, to=7-4]
	\arrow[bend left=24, from=5-5, to=7-5]
	\arrow[from=5-6, to=6-5]
	\arrow[bend left=24, from=5-6, to=7-6]
	\arrow[from=6-2, to=7-2]
	\arrow[bend right=24, from=6-2, to=8-2]
	\arrow[from=6-3, to=7-4]
	\arrow[bend right=24, from=6-3, to=8-3]
	\arrow[from=6-4, to=6-5]
	\arrow[bend right=24, from=6-4, to=8-4]
	\arrow[bend right=24, from=6-5, to=8-5]
	\arrow[from=6-6, to=7-4]
	\arrow[bend right=24, from=6-6, to=8-6]
	\arrow[from=7-1, to=9-2]
	\arrow[bend left=24, from=7-2, to=9-2]
	\arrow[from=7-3, to=8-3]
	\arrow[from=7-4, to=7-5]
	\arrow[bend left=24, from=7-4, to=9-4]
	\arrow[bend left=24, from=7-5, to=9-5]
	\arrow[from=7-6, to=8-5]
	\arrow[from=8-2, to=9-2]
	\arrow[from=8-3, to=9-4]
	\arrow[from=8-4, to=8-5]
	\arrow[from=8-6, to=9-4]
	\arrow[from=9-4, to=9-5]
\end{tikzcd}
\end{adjustbox}
\caption{Set theories mutually interpretable with \(\IS{n}\) for \(n\geq1\)}
\end{figure}
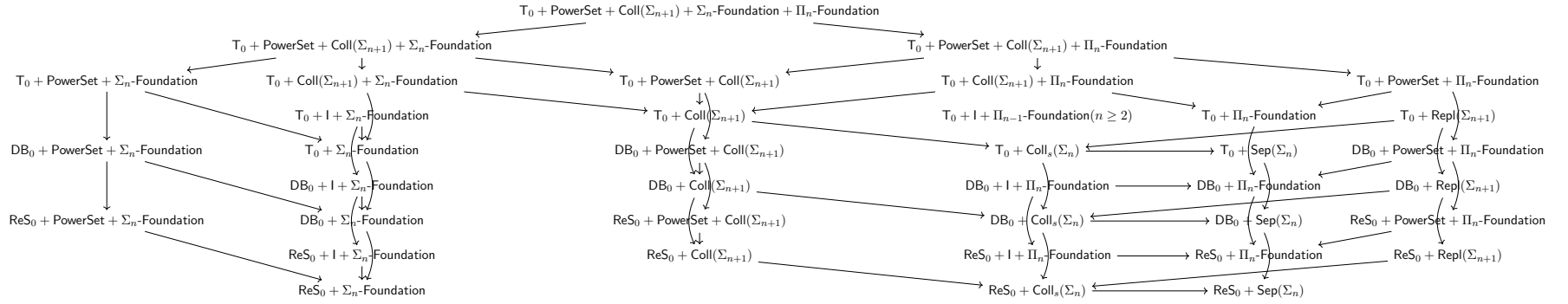
\vspace*{\fill}
\end{landscape}

These diagrams may also be extended to higher levels. One direction is to study systematically stronger set theories with Infinity but without Power Set, and to compare them with second-order arithmetic. Another is to adjust first-order arithmetical theories so as to suit ordinal arithmetic, and thereby to find arithmetical languages appropriate for comparison with the corresponding set theories; such systems should also be mutually interpretable with second-order arithmetic. Set theories possessing both Infinity and Power Set can no longer be called weak, and raise a third family of questions: how their ordinary interpretability degrees vary with Separation, Collection, Replacement, and Foundation, and at what point they begin to approach the usual theory \(\mathsf{ZFC}\). If suitable second-order theories of ordinal arithmetic and their variants are considered, they should also lie in this higher region.

\bibliographystyle{plainnat}
\bibliography{ref}

\end{document}